\documentclass[12 pt,english]{article}
\usepackage[activeacute]{babel}
\usepackage{amsfonts} 
\usepackage{mathrsfs} 
\usepackage{pstricks-add}
\usepackage{amssymb, amsmath, amsfonts}
\usepackage{makeidx}
\usepackage{epsfig}
\usepackage{color}
\usepackage[T1]{fontenc}
\usepackage{pstricks, pst-node}
\usepackage{graphics,graphicx,graphpap}
\usepackage[ansinew]{inputenc}
\usepackage{fancyhdr}
\usepackage{color}
\usepackage{amsthm}
\usepackage{multirow}

\if@twoside \oddsidemargin 5pt \evensidemargin 5pt \marginparwidth 20pt
 \else \oddsidemargin 5pt \evensidemargin 5pt \marginparwidth 20pt \fi

\newcommand{\ZZ}{\mathbb{Z}}

\newcommand{\Aut}{\mathrm{Aut}}
\newcommand{\G}{\Gamma}
\newcommand{\U}{\Upsilon}
\newcommand{\tG}{\tilde{\Gamma}}
\renewcommand{\O}{\mathcal{O}}

\renewcommand{\P}{\mathcal{P}}
\newcommand{\Cay}{\mathop{\rm Cay}}
\renewcommand{\S}{\mathrm{Sym}}
\newcommand{\cB}{\mathcal{B}}

\newcommand{\Ml}{\mathop{\rm Ml}}
\renewcommand{\Pr}{\mathop{\rm Pr}}
\newcommand{\GPr}{\mathop{\rm GPr}}

\newtheorem{theorem}{Theorem}[section]
\newtheorem{proposition}[theorem]{Proposition}
\newtheorem{corollary}[theorem]{Corollary}
\newtheorem{lemma}[theorem]{Lemma}
\newtheorem{problem}[theorem]{Problem}

\theoremstyle{definition}
\newtheorem*{remark}{Remark}
\newtheorem{construction}[theorem]{Construction}
\newtheorem{example}[theorem]{Example}

\begin{document}

\begin{center}
\Large{\textbf{Symmetry properties of generalized graph truncations}} \\ [+4ex]
Eduard Eiben{\small$^{a,}$\footnotemark},\ Robert Jajcay{\small$^{b, c, }$\footnotemark$^{,*}$} and Primo\v z \v Sparl{\small$^{c, d, e,}$\footnotemark}
\\ [+2ex]
{\it \small 
$^a$Algorithms and Complexity Group, TU Wien, Vienna, Austria\\
$^b$Comenius University, Bratislava, Slovakia\\
$^c$University of Primorska, Institute Andrej Maru\v si\v c, Koper, Slovenia\\
$^d$University of Ljubljana, Faculty of Education, Ljubljana, Slovenia\\
$^e$Institute of Mathematics, Physics and Mechanics, Ljubljana, Slovenia\\}
\end{center}

\addtocounter{footnote}{-2}
\footnotetext{Supported in part by the Austrian Science Fund (FWF), projects P26696 and W1255-N23.}
\addtocounter{footnote}{1} 
\footnotetext{Supported in part by VEGA 1/0474/15, VEGA 1/0596/17, NSFC 11371307, APVV-15-0220, and by the Slovenian Research Agency research project J1-6720.}
\addtocounter{footnote}{1} 
\footnotetext{Supported in part by the Slovenian Research Agency, program P1-0285 and projects N1-0038, J1-6720 and J1-7051. 

Email addresses: 
eiben@ac.tuwien.ac.at (Eduard Eiben),
Robert.Jajcay@fmph.uniba.sk (Robert Jajcay),
primoz.sparl@pef.uni-lj.si (Primo\v z \v Sparl).

~* corresponding author  }

%%%%%%%%%%%%%%%%%
%%%		Abstract
%%%%%%%%%%%%%%%%%

\hrule

\begin{abstract}
In the generalized truncation construction, one replaces each vertex 
of a $k$-regular graph $\G$ with a copy of a graph $\U$ of order $k$. 
We investigate the symmetry properties of the graphs constructed
in this way, especially in connection to the symmetry properties of the graphs
$\G$ and $\U$ used in the construction. We demonstrate the usefulness of our results 
by using them to obtain
a classification of cubic vertex-transitive graphs of girths $3$, $4$, and $5$. 
\end{abstract}

\hrule

\begin{quotation}
\noindent {\em \small Keywords: truncation, automorphism group, vertex-transitive
graph, girth}
\end{quotation}

%%%%%%%%%%%%%%%%%%%%%%%%%%%%%%
%----------------Introduction---------------
%%%%%%%%%%%%%%%%%%%%%%%%%%%%%%
\section{Introduction}

The original concept of graph truncations comes from topological 
graph theory. The concept was repeatedly recycled in different 
settings, as in Sachs' classical article \cite{sachs}, in which
it has been used to prove the existence of $k$-regular graphs of girth
$g$ for every pair $ k \geq 3,g \geq 3 $. Its generalization appears under the name
of a zig-zag product in a number of articles dealing with graph
expanders (see, for example, \cite{AloLubWig01,reivad&wig}). The second author 
together with G. Exoo generalized
the original construction of Sachs to allow for truncations
by any graph of the correct order (while Sachs only used truncations by
cycles) \cite{ExoJaj12}. They used the 
generalized truncation construction for constructing small graphs of given
degree and girth.
Several forms of generalized truncations also
appear in \cite{BobJaj&Pis}, which uses them to construct
graphs with prescribed degrees and prescribed beginning of the cycle spectrum.

The paper \cite{ExoJaj12} 
contains several examples
of truncations of arc-transitive maps (which, of course, have arc-transitive
underlying graphs) that result in
vertex-transitive graphs.
% -- a fact that becomes best understood in 
%the light of results presented in here. 
The second source of 
interest and motivation for our paper comes from \cite{AlsDob15}, 
in which the authors investigate the symmetry and hamiltonicity of
generalized truncation of graphs via complete graphs $K_n$. 
As both papers \cite{ExoJaj12} and \cite{AlsDob15} start off with a symmetry assumption 
about the underlying graph or the graph that is being attached
to the underlying graph, we investigate the general symmetry properties of 
the graphs resulting from generalized truncation and the impact of the 
symmetry properties of the graphs used in the construction on the symmetry
properties of the resulting graphs. In this sense, our results can be viewed as 
generalizations of the results from \cite{ExoJaj12} and \cite{AlsDob15}.

In addition to being able to obtain results
concerning the symmetry of the graphs resulting from the generalized
truncation, we demonstrate the strength of these ideas
by investigating the class of cubic vertex-transitive graphs of girth $3,4$ and $5$, 
for which we obtain a full characterization.

%%%%%%%%%%%%%%%%%%%%%%%%%%%%%%
%----------------Preliminaries-----------------------------
%%%%%%%%%%%%%%%%%%%%%%%%%%%%%%
\section{Preliminaries}
\label{sec:prelim}

Graphs used in our paper are finite and simple. 
Given a graph $\G$, we denote its vertex-set and edge-set by $V(\G)$ and $E(\G)$, respectively. For a vertex $v \in V(\G)$, we let $\G(v) = \{u \in V(\G) \mid u \sim v\}$ denote its neighborhood.

A graph $\G$ is said to be vertex-, edge-, or arc-transitive, if the automorphism
group of $\G$ acts transitively on the set of vertices, edges, or arcs of $\G$.
For a group $G$ and a subset $S \subset G$ closed under inverses but not 
containing the identity, the {\em Cayley graph} $\Cay(G;S)$ of the group $G$ with respect to the {\em connection set} $S$ is the graph with vertex set $G$ in which 
a vertex $g \in G$ is adjacent to a vertex $ h \in G $ if and only if $h = sg$ for some $s \in S$. All Cayley graphs are regular (of degree $|S|$) and vertex-transitive; with the underlying group
$G$ acting vertex-transitively via automorphisms induced by 
the left multiplication by its elements.

The generalized truncation construction also assumes regularity of one of the 
two graphs involved.  
We repeat the definition of the generalized truncation from~\cite{ExoJaj12}. Let $\G$ be a finite $k$-regular graph and let $D(\G)$ denote the set of its darts (that is, ordered pairs of adjacent vertices). A {\em vertex-neighborhood labeling} of $\G$ is a function $\rho \colon D(\G) \to \{1,2,\ldots , k\}$ such that for each $u \in V(\G)$ the restriction of $\rho$ to the set $\{(u,v)\colon v \in \G(u)\}$ of darts emanating from $u$ is a bijection. Furthermore let $\U$ be a graph of order $k$ with $V(\U) = \{v_1, v_2, \ldots , v_{k}\}$. The {\em generalized truncation} $T(\G, \rho; \U)$ of $\G$ by $\U$ with respect to $\rho$ is then the graph with the vertex set $\{(u,v_i) \colon u \in V(\G), 1 \leq i \leq k\}$ and edge set 
\begin{equation}\label{def}
 \{ \; (u,v_i) (u,v_j) \; | \; v_iv_j \in E(\U) \; \} \cup \{ \; (u,v_{\rho(u,w)}) (w,v_{\rho(w,u)}) \; | \; uw \in E(\G) \; \} .
\end{equation}
Informally, the graph $T(\G, \rho; \U)$ is obtained from the $k$-regular $\G$ by 
`cutting out' each vertex $u$ together with a small part of its neighborhood and `glueing in' copies of $\U$ in such a way that the vertex $v_i$ of $\U$ is attached to the dangling dart previously attached to $u$ and labeled $i$. 
Note that $T(\G,\rho;\U)$ is regular if and only if $\U$ is regular.

We will refer to the edges of the form $(u,v_i)(u,v_j)$ (edges originally contained in a copy of $\U$) as {\em red} edges, and to the edges $(u,v_i)(w,v_j)$, $u \neq w$, (originally contained in $\G$) as {\em blue} edges. Observe that a blue edge is incident only to red edges and that each vertex of $T(\G,\rho;\U)$ is incident to precisely one blue edge. This yields the following observation (see also~\cite[Theorem~2.2]{ExoJaj12}).

\begin{lemma}
\label{le:blue_cycle}
Let $\G$ be a finite $k$-regular graph of girth $g$. Then, for any graph $\U$ of order $k$ and any vertex-neighborhood labeling $\rho$ of $\G$, the shortest cycle in the generalized truncation $T(\G,\rho;\U)$ containing a blue edge is of length at least $2g$.
\end{lemma}

%%%%%%%%%%%%%%%%%%%%%%%%%%%%%%
%----------------  Basic properties  -----------------------------
%%%%%%%%%%%%%%%%%%%%%%%%%%%%%%
\section{Basic symmetry properties of generalized truncations}
\label{sec:basic}

One of the main objectives of this paper is to investigate automorphisms of generalized truncations of graphs and their relationship to the automorphisms of 
the underlying truncated graph. Let $\tilde{\G} = T(\G,\rho;\U)$ be the generalized truncation of $\G$ by $\U$ with respect to $\rho$ and let $u \in V(\G)$ be a vertex of $\G$. We let $\tilde{\G}_u$ denote the subgraph of $\tilde{\G}$ induced on the set $\{(u,v_i) \colon 1 \leq i \leq k\}$, that is, $\tilde{\G}_u$ is the copy of $\U$ that replaced the vertex $u$ of $\G$. This gives rise to the {\em natural partition} $\P_\G = \{V(\tilde{\G}_u) \colon u \in V(\G)\}$ of the vertex set of $\tG$, which plays an important role in our investigation of automorphisms of $\tG$.

\begin{lemma}
\label{le:kernel}
Let $\tG = T(\G,\rho;\U)$ be a generalized truncation. The only automorphism of $\tG$ fixing each partition set in $\P_\G$ setwise is the identity.
\end{lemma} 

\begin{proof}
Observe that for any $u,w \in V(\G)$ there is either exactly one blue edge or there are no blue edges between the vertices of $\tG_u$ and $\tG_w$, depending on whether $u$ and $w$ are or are not adjacent in $\G$. This means that each $\alpha \in \Aut(\tG)$ fixing setwise each of the sets $\tG_u$, fixes pointwise all the endpoints of the blue edges. Since each vertex of $\tG$ is the endpoint of exactly one blue edge, this proves that $\alpha$ is the identity.
\end{proof}

We can now prove that each subgroup $\tilde{G}$ of $\Aut(\tG)$ leaving $\P_\G$ invariant 
is  isomorphic to a subgroup of $\Aut(\G)$.

\begin{theorem}
\label{the:projects}
Let $\tG = T(\G,\rho;\U)$ be a generalized truncation, and let $\P_\G$ be the natural partition of the vertex set of $\tG$. Let $\tilde{G} \leq \Aut(\tG)$ be any subgroup leaving $\P_\G$ invariant. Then $\tilde{G}$ induces a natural faithful action on $\G$ and is thus isomorphic to a subgroup of $\Aut(\G)$. 
\end{theorem}

\begin{proof}
Let $\tilde{g} \in \tilde{G}$ be arbitrary and let us define a corresponding automorphism $g$ of $\G$. For each $u \in V(\G)$ let $u^g \in V(\G)$ be the unique vertex such that $\tilde{g}$ maps $V(\tG_u)$ to $V(\tG_{u^g})$. Since $\P_\G$ is $\tilde{G}$-invariant, $g$ is a well defined permutation of $V(\G)$. In view of Lemma~\ref{le:kernel}, we only have to prove that $g$ is an automorphism of $\G$. To this end suppose $u$ and $w$ are adjacent vertices of $\G$ and let $(u,v_i)$ and $(w,v_j)$ be the unique vertices of $\tG_u$ and $\tG_w$, respectively, such that $(u,v_i) \sim (w,v_j)$ in $\tG$. Since $(u^g, v_{i'}) = (u,v_i)^{\tilde{g}} \sim (w,v_j)^{\tilde{g}} = (w^g, v_{j'})$ for some $i', j'$, it is clear that $u^g \sim w^g$ in $\G$.
\end{proof}

If $g$ is constructed from $\tilde{g}$ as in the proof of the above theorem, we will say that $g \in \Aut(\G)$ is the {\em projection} of $\tilde{g} \in \Aut(\tG)$ and conversely that $\tilde{g}$ is the {\em lift} of $g$ (observe that, by Lemma~\ref{le:kernel}, each $g \in \Aut(\G)$ can have at most one lift to $\Aut(\tG)$). We will also say that $g$ {\em lifts} to an automorphism of $\tG$ and $\tilde{g}$ {\em projects} onto an automorphism of $\G$. 
Of course, $\Aut(\tG)$ may also contain automorphisms which do not project to $\Aut(\G)$ (we will call such automorphisms {\em mixers}) and there can be automorphisms of $\G$ which do not lift to $\Aut(\tG)$. 
As a consequence of Theorem~\ref{the:projects}, an automorphism $\tilde{g}$ of $\tG$ projects to $\Aut(\G)$ if and only if the partition $\P_\G$ is $\langle \tilde{g} \rangle$-invariant. This observation gives an easy sufficient condition for the entire group $\Aut(\tG)$ to project to $\Aut(\G)$.

\begin{corollary}
\label{cor:small_girth}
Let $\G$ be a $k$-regular graph of girth $g$, and $\tG = T(\G,\rho;\U)$ be a generalized truncation. If $\U$ is connected and each of its edges lies on at least one cycle of length smaller than $2g$, then the entire automorphism group $\Aut(\tG)$ projects injectively onto a subgroup of $\Aut(\G)$.
\end{corollary} 

\begin{proof}
By Lemma~\ref{le:blue_cycle}, any cycle of $\tG$ containing a blue edge is of length at least $2g$. Since each red edge belongs to a cycle that is shorter than $2g$, the partition $\P_\G$ must be $\Aut(\tG)$-invariant, and the result follows from Theorem~\ref{the:projects}.
\end{proof}

There are many classes of graphs having the property that each edge of the graph is contained in a
short cycle. For example, this must be the case for arc-transitive graphs of small girth, or more generally,
for {\em edge-girth-regular graphs}, which are graphs in which each edge is contained in the same 
number of girth cycles \cite{JajKisMik17}, as well as  for Cayley graphs of abelian or nilpotent groups \cite{ConExoJaj10}.

\begin{corollary}
\label{cor:central}
Let $\U$ be a connected Cayley graph $\Cay(G;S)$ satisfying the property that $S$ contains at least three elements
out of which at least one belongs to the center $Z(G)$, and let $\tG = T(\G,\rho;\U)$ be a generalized truncation.
Then the entire automorphism group $\Aut(\tG)$ projects injectively onto a subgroup of $\Aut(\G)$.
\end{corollary}

\begin{proof}
While the girth $g$ of any $|G|$-regular graph $\G$ is at least $3$, the 
assumptions imply that each edge of $\U$ lies on a $3$-cycle or a $4$-cycle, and so Corollary~\ref{cor:small_girth} applies.
\end{proof}
\medskip

Having addressed the question of when an automorphism of $\tG = T(\G,\rho;\U)$ projects, let us now consider the question of when an automorphism $g \in \Aut(\G)$ lifts to some $\tilde{g} \in \Aut(\tG)$. If such a lift $\tilde{g}$ exists, it is unique by Lemma~\ref{le:kernel}. In fact, we now describe how the action of $\tilde{g}$ (if it exists) is determined by $g$. 
Let $(u,v_i)$ be an arbitrary vertex of $\tG$ and let $(w,v_j)$ be its unique neighbor along a blue edge, that is $i = \rho(u,w)$ and $j = \rho(w,u)$,
and $u$ and $w$ are adjacent in $\G$. 
Since $g \in \Aut(\G)$, the vertices $u^g$ and $w^g$ must also be adjacent in $\G$, and so $(u^g, v_{\rho(u^g, w^g)})$ is adjacent to $(w^g, v_{\rho(w^g,u^g)})$ by the definition of $\tG$. If a lift $\tilde{g}$ of $g$ existed, it would have to send $(u,v_i) = (u,v_{\rho(u,w)})$ to $(u^g, v_{\rho(u^g, w^g)})$ and $(w,v_j) = (w,v_{\rho(w,u)})$ to $(w^g,v_{\rho(w^g,u^g)})$. In this manner, each automorphism $g \in \Aut(G)$ 
yields a unique permutation $\tilde{g}$ of the vertices of $\tG$ which maps blue edges to blue edges. The permutation $\tilde{g}$ may or may not be
a graph automorphism of $\tG$, depending on whether it also maps red 
edges to red edges.
This provides us with a necessary and sufficient condition for an automorphism of $\G$ to lift to $\Aut(\tG)$.

\begin{proposition}
\label{pro:lifts}
Let $\tG = T(\G,\rho;\U)$ be a generalized truncation, and let $g \in \Aut(\G)$. Then $g$ lifts to $\Aut(\tG)$ if and only if for every $u \in V(\G)$ and each pair of its neighbors $w,x$ we have 
\begin{equation}\label{eq:lifts}
	v_{\rho(u,x)} \sim v_{\rho(u,w)} \iff v_{\rho(u^g,x^g)} \sim v_{\rho(u^g,w^g)} \ \text{in}\ \U.
\end{equation}
As a consequence, the set of all $g \in \Aut(\G)$ that lift to $\Aut(\tG)$ is a subgroup of $G$.
\end{proposition}

Observe that in the case when $\U$ is a complete graph, the condition (\ref{eq:lifts}) is automatically satisfied for all $ g \in \Aut(\G)$, and so the entire automorphism group of $\G$ lifts (as proved in \cite[Theorem~3.6]{AlsDob15}). We finish this section with an illustration of the above results by means of a concrete example.

\begin{example}
\label{ex:K_5_trunc_C_4}
We consider two non-isomorphic generalized truncations of the complete graph $K_5$ by a $4$-cycle $C_4$. The corresponding vertex neighborhood labelings $\rho_1$ and $\rho_2$ are given in Figure~\ref{fig:K_5trunc1}, where the vertex set of $K_5$ consists of the elements $\{a,b,c,d,e\}$.
\begin{figure}[!h]
\begin{center}
\includegraphics[scale=0.6]{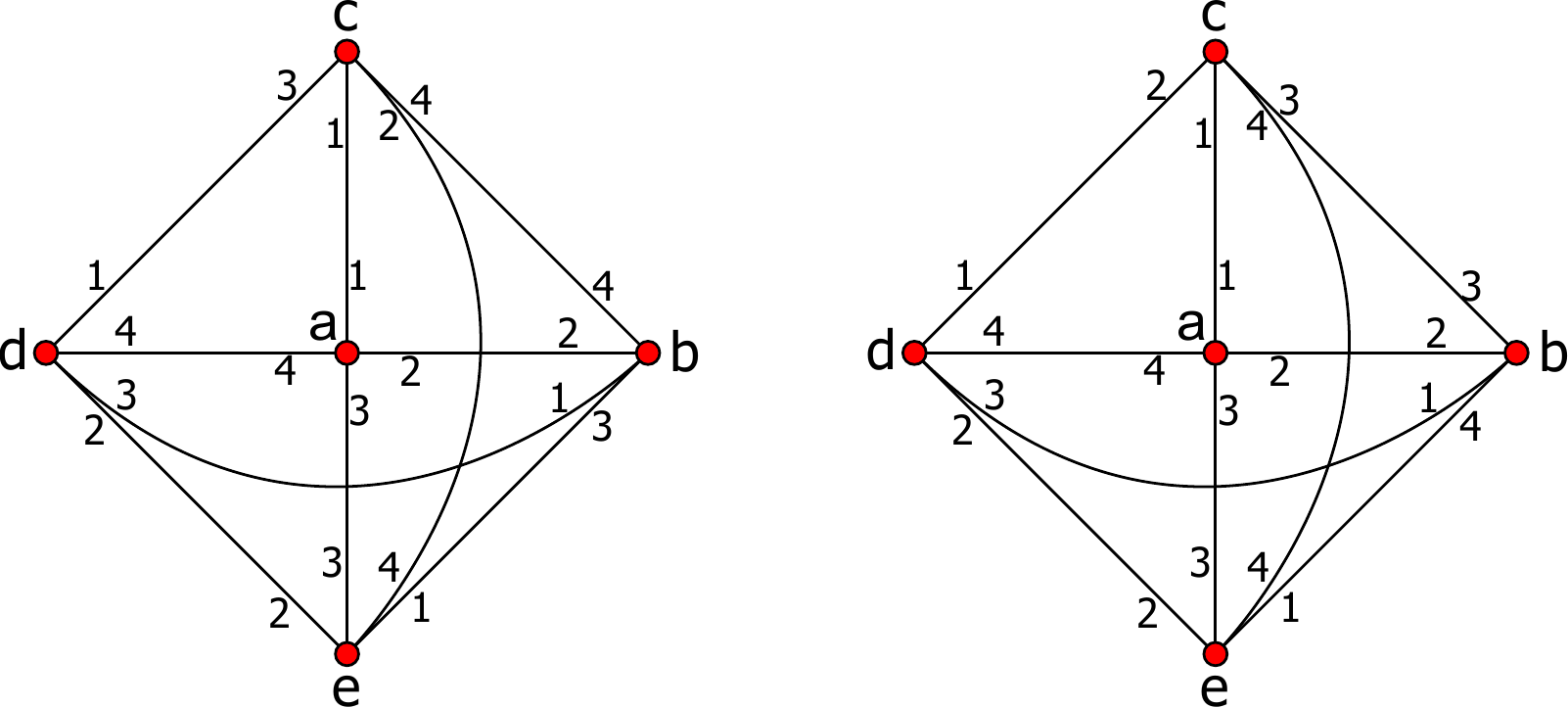}
\end{center}
\caption{The vertex neighborhood labelings $\rho_1$ and $\rho_2$ of $K_5$.}
\label{fig:K_5trunc1}
\end{figure}

If we denote the vertex set of $C_4$ by $\{v_1,v_2,v_3,v_4\}$, with the antipodal pairs of vertices being $\{v_1,v_3\}$ and $\{v_2,v_4\}$, 
the two corresponding generalized truncations are given in Figure~\ref{fig:K_5trunc2}. 
\begin{figure}[!h]
\begin{center}
\includegraphics[scale=0.4]{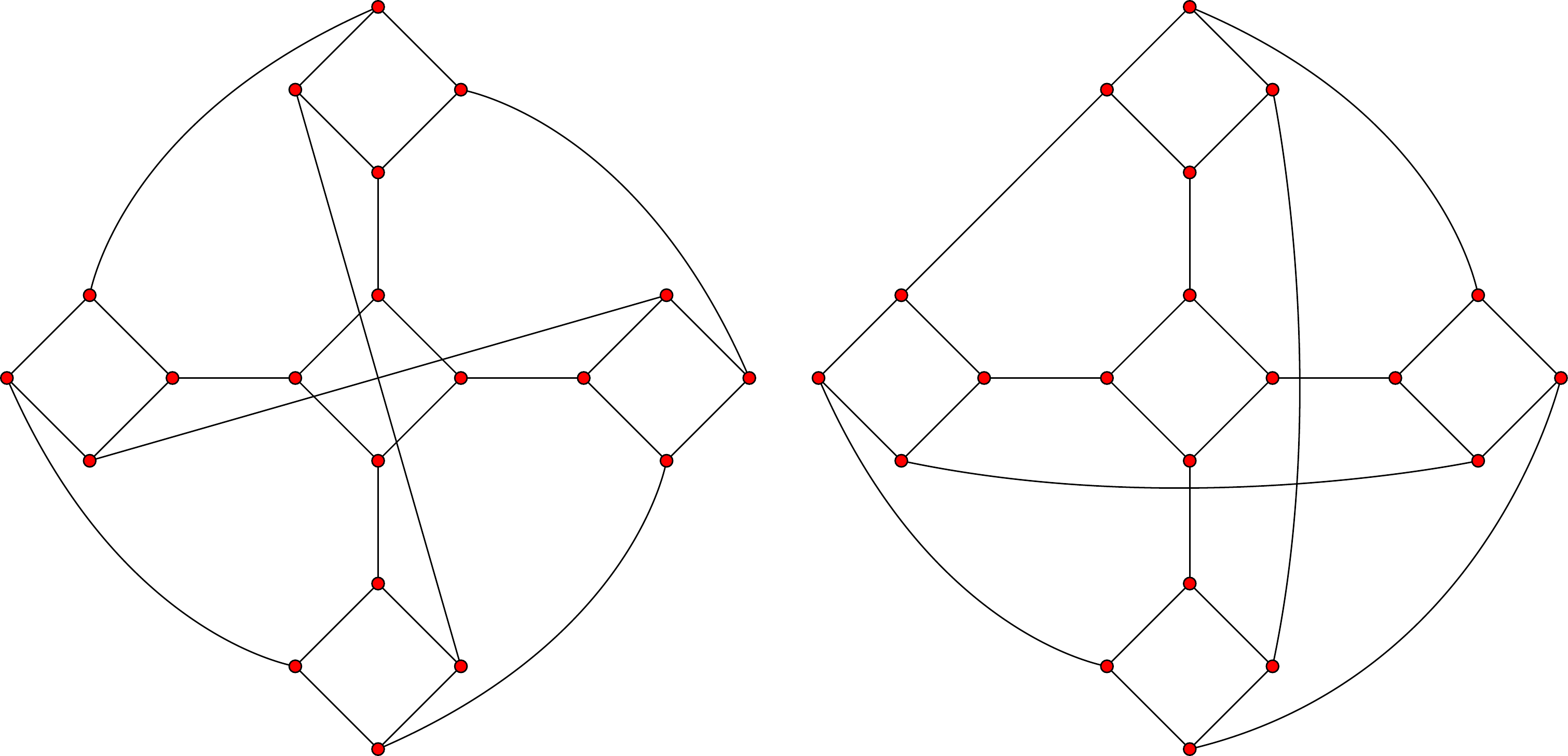}
\end{center}
\caption{Two nonisomorphic generalized truncations of $K_5$ by $C_4$.}
\label{fig:K_5trunc2}
\end{figure}
It is easy to see that the two graphs are non-isomorphic, since the first one does not contain any $6$-cycles, while the second one does. To investigate which of the automorphisms of $K_5$ lift in either of the two generalized
truncations, let $\tG_i = T(K_5,\rho_i; C_4)$, $i \in \{1,2\}$. Observe first, that since both $\tG_i$'s are truncations by a $4$-cycle, Corollary~\ref{cor:small_girth} implies that, in either case, the whole automorphism group $\Aut(\tG_i)$ projects to $\Aut(K_5) = \mathrm{Sym}_5$.
Therefore, $\Aut(\tG_i)$, $i \in \{1,2\}$, coincides with the set of lifts of those 
automorphisms of $K_5$ that do lift with respect to the corresponding 
vertex neighborhood labeling.

We first consider the generalized truncation $\tG_1$. It is not hard to check that the automorphisms $(a\,b\,c\,e\,d)$ and $(b\,c\,d\,e)$ of $K_5$ both satisfy the conditions of Proposition~\ref{pro:lifts}, and so the subgroup $G = \langle (a\,b\,c\,e\,d), (b\,c\,d\,e) \rangle \cong \mathrm{AGL}_1(5)$ of the symmetric group $\mathrm{Sym}_5$ (which is a maximal subgroup of $\mathrm{Sym}_5$) lifts to $\Aut(\tG_1)$. In particular, this implies that
$\tG_1$ is vertex-transitive. On the other hand, the automorphism $g = (a\,b)$ of $K_5$ does not lift to $\Aut(\tG_1)$, since, for instance, $v_{\rho_1(a,c)} = v_1 \sim v_2 = v_{\rho_1(a,b)}$, but $v_{\rho_1(a^g,c^g)} = v_{\rho_1(b,c)} = v_4$ is not adjacent to $v_{\rho_1(a^g,b^g)} = v_{\rho_1(b,a)} = v_2$. This 
proves that the above subgroup $G$ of $\mathrm{Sym}_5$ is the maximal subgroup of the automorphism group of $K_5$ that lifts to $\Aut(\tG_1)$, and so $\tilde{G} = \Aut(\tG_1) \cong \mathrm{AGL}_1(5)$. Note that this implies $\tG_1$ is a Cayley graph of $\mathrm{AGL}_1(5)$.

As for the generalized truncation $\tG_2$, it is easy to check again that the automorphism $(b\,e)(c\,d)$ of $K_5$ satisfies the conditions of Proposition~\ref{pro:lifts}, and hence it lifts to $\Aut(\tG_2)$. On the other hand, none of the automorphisms $(b\,e), (b\,c\,e\,d), (b\,c)(d\,e), (b\,e\,c)$, $(a\,b\,c\,d\,e)$, $(a\,b\,d\,c\,e)$, $(a\,b\,e)$ and $(a\,c\,d)$ satisfies the conditions of Proposition~\ref{pro:lifts}, proving that $\Aut(\tG_2) \cong \ZZ_2$. \hfill $\blacktriangle$
\end{example}

%%%%%%%%%%%%%%%%%%%%%%%%%%%%%%
%----------------  Constructions from vertex-transitive graphs  -----------------------------
%%%%%%%%%%%%%%%%%%%%%%%%%%%%%%
\section{A construction from vertex-transitive graphs}
\label{sec:trunc_from_VT}

In this section, we present a method for constructing generalized truncations which we will later prove in Theorem~\ref{the:VTfromAT} to yield a vertex-transitive graph whenever the truncated graph happens to be arc-transitive.

\begin{construction}
\label{con:T}
Let $\G$ be a graph admitting a vertex-transitive subgroup $G \leq \Aut(\G)$ of automorphisms. For a fixed vertex $v \in V(\G)$, let $\O_v$ be a union of orbits of the action of the stabilizer $G_v$ in its induced action on the 
$2$-element subsets of $\G(v)$. Then $\O_v$, which consists of unordered
pairs of vertices from $\G(v)$, forms a set of edges for 
the graph $(\G(v),\O_v)$, whose vertex set consists of the vertices from $\G(v)$. We define the graph $T(\G,G;\O_v)$ to be the graph with the vertex set $\{(u,w) \colon u \in V(\G), w \in \G(u)\}$, and the adjacency relation in which a vertex $(u,w)$ is adjacent to the vertex 
$(w,u)$ and to all the vertices $(u,w')$ for which there exists a $g \in G$ with the property $u^g = v$ and $\{w,w'\}^g \in \O_v$. 
\end{construction}

Even though the above defined graph $T(\G,G;\O_v)$ might appear different
from the generalized truncation graphs considered so far, in what follows,
we will show that it indeed is a generalized truncation of a vertex-transitive
$\G$ via the orbital graph $ (\G(v),\O_v)$. 

Let us begin by pointing out that a vertex $(u,w)$ is adjacent to $(u,w')$ in $T(\G,G;\O_v)$ if and only if $\{w,w'\}^g \in \O_v$ holds {\em for each} $g \in G$ satisfying $u^g = v$. Namely, if $g, g' \in G$ both satisfy $u^g = v = u^{g'} $, then $g^{-1}g' \in G_v$, and so the fact that $\O_v$ is a union of orbits of the action of $G_v$ on the set of $2$-element subsets of $\G(v)$ implies that $\{w^g,w'^g\} \in \O_v$ if and only if $\{w^g,w'^g\}^{g^{-1}g'} = \{w^{g'}, w'^{g'}\} \in \O_v$.

\begin{lemma}
\label{le:trunc_from_VT}
Let $\G$ be a vertex-transitive graph, and let $G \leq \Aut(\G)$, $v \in V(\G)$, and $\O_v$, be as those in Construction~\ref{con:T}. Then the graph $T(\G,G;\O_v)$ is isomorphic to a generalized truncation of the graph $\G$ by the graph $(\G(v),\O_v)$. 
\end{lemma}

\begin{proof}
Let $k = |\Gamma(v)|$ be the cardinality of $\G(v)$, and choose an arbitrary bijection $\varphi \colon \Gamma(v) \to \{1,2,\ldots , k\}$.
% (that is, enumerate the vertices of $\Gamma(v)$). 
Take $\U$ to be the graph with the vertex set $\{v_1, v_2, \ldots , v_k\}$ 
and the adjacency $v_i \sim v_j$ if and only if $ \{ \varphi^{-1}(i),\varphi^{-1}(j)
\} \in \O_v$. Clearly, $\U$ is isomorphic to the orbital graph $(\G(v),\O_v)$.
For each $u \in V(\Gamma)$, choose a $g_u \in G$ with the property
$u^{g_u} = v$. For each pair of adjacent vertices $u,w$ of $\Gamma$ we then set $\rho(u,w) = \varphi(w^{g_u})$. Since $w \sim u$ and $u^{g_u} = v$, we have $w^{g_u} \in \G(v)$, and so the expression $  \varphi(w^{g_u}) $ `makes sense' and these
assignments define a vertex-neighborhood labeling for $\Gamma$. 

We can now define a mapping $\Phi \colon V(T(\G,G;\O_v)) \to V(T(\G,\rho;\U))$ by setting $\Phi(u,w) = (u,v_{\rho(u,w)})$ for each vertex $(u,w)$ of $T(\G,G;\O_v)$. It is easy to check that $\Phi$ is a bijection preserving adjacency. For instance, if $(u,w) \sim (u,w')$ in $T(\G,G;\O_v)$, then by definition and the above remarks, $\{w,w'\}^{g_u} \in \O_v$. Thus $\Phi(u,w) = (u,v_{\rho(u,w)})$ and $\Phi(u,w') = (u,v_{\rho(u,w')})$ are adjacent in $T(\G,\rho;\U)$. This proves that  $T(\G,G;\O_v) \cong T(\G,\rho;\U)$, as claimed.
\end{proof}

In line with the terminology of Section~\ref{sec:basic}, the edges of $\tG = T(\G,G;\O_v)$ of the form $(u,w)(w,u)$ will be called `blue edges' and the edges of the form $(u,w)(u,w')$ will be called `red'. As pointed out in Section~\ref{sec:basic}, $\tG$ is regular
if and only if $(\G(v),\O_v)$ is regular. A sufficient (but not necessary) condition for 
$(\G(v),\O_v)$ being regular is arc-transitivity of the action of $G$ on $\G$. Note also that the elements of the partition $\P_\G$ are $V(\tilde{\G}_u) = \{(u,w)\colon w \in \G(u)\}$. It is now relatively easy to determine which automorphism groups of $\G$ lift to $\Aut(\tG)$.

\begin{proposition}
\label{pro:liftsVT}
Let $\G$, $G$, $\O_v$, and $\tG = T(\G,G;\O_v)$, be as in Construction~\ref{con:T}. 
A subgroup $K \leq \Aut(\G)$, containing $G$, lifts to $\Aut(\tG)$ if and only if $\O_v$ is a union of orbits of the action of $K_v$ on the $2$-sets of elements from 
$\G(v)$. In particular, $G$ itself lifts to $\Aut(\tG)$. 
\end{proposition}

\begin{proof}
By Lemma~\ref{le:kernel} and Theorem~\ref{the:projects}, it suffices to determine whether for each $g \in K$ there is an automorphism $\tilde{g} \in \Aut(\tG)$ such that $\P_\G$ is $\langle \tilde{g}\rangle$-invariant and the induced action of $\tilde{g}$ on $\P_\G$ coincides with the action of $g$ on $\G$. Note that the action of a potential $\tilde{g}$ on $\tG$ is uniquely determined and is very natural. For each vertex $(u,w)$ of $\tG$ we need to have $(u,w)^{\tilde{g}} = (u^g,w^g)$. Since $g$ is an automorphism of $\G$, $\tilde{g}$ constructed in this way is clearly a permutation of the vertex set of $\tG$ preserving the set of the blue edges of $\tG$. 

It thus suffices to determine whether it also preserves the red edges. If $\O_v$ is not a union of orbits of the action of $K_v$ on the $2$-sets of elements from $\G(v)$, then there exists some $g \in K_v$ and some $\{w,w'\} \in \O_v$ for which $\{w,w'\}^g \notin \O_v$. In this case, $\tilde{g}$ is not an automorphism of $\tG$, and $K$ does not lift to $\Aut(\tG)$. Suppose next that $\O_v$ is a union of orbits of the action of $K_v$ on the $2$-sets of elements of $\G(v)$, $g \in K$, $(u,w)(u,w')$ is a red edge of $\tG$, and $h \in G$ is such that $u^h = v$ (and consequently $\{w,w'\}^h \in \O_v$). Since $G$ is transitive on $V(\G)$, there exists some $h' \in G$ such that $u^{gh'} = v$. Then $h^{-1}gh' \in K_v$ (recall that $G \leq K$). Since $\O_v$ is a union of orbits of the action of $K_v$ and $\{w^h,w'^h\} \in \O_v$, it follows that $\{w^g, w'^g\}^{h'} = \{w^h, w'^h\}^{h^{-1}gh'} \in \O_v$, which proves that the vertices $(u,w)^{\tilde{g}}$ and $(u,w')^{\tilde{g}}$ are adjacent in $\tG$.
\end{proof}

%Next, we state a necessary and sufficient condition for an automorphism of $\G$ to lift.

\begin{corollary}
\label{cor:liftsVT}
Let $\G$, $G$, $\O_v$, and $\tG = T(\G,G;\O_v)$, be as in Construction~\ref{con:T}. An automorphism $h$ of $\G$ lifts to $\tG$ if and only if $\O_v$ is a union of orbits of the action of $\langle G, h \rangle_v$ on the $2$-sets of elements from $\G(v)$.
\end{corollary}

%%%%%%%%%%%%%%%%%%%%%%%%%%%%%%
%----------------  Vertex-transitive generalized truncations  -----------------------------
%%%%%%%%%%%%%%%%%%%%%%%%%%%%%%
\section{Vertex-transitive generalized truncations}
\label{sec:VTtrunc}

The results of the previous section yield a very natural construction for vertex-transitive generalized truncations. 

\begin{theorem}
\label{the:VTfromAT}
Let $\G$ be an arc-transitive graph, and let $G \leq \Aut(\G)$ be an arc-transitive 
group of automorphisms. Let $v$ be a vertex of $\G$, let $\O_v$ be a union of orbits of the action of $G_v$ on the $2$-sets of elements from $\G(v)$, and let $\tG = T(\G,G;\O_v)$ be the corresponding generalized truncation. Then $G$ lifts to $\tilde{G} \leq \Aut(\tG)$ which acts vertex-transitively on $\tG$. 
\end{theorem}

\begin{proof}
By Proposition~\ref{pro:liftsVT}, the subgroup $G$ lifts to $\Aut(\tG)$. Since the action of $G$ on $\G$ is arc-transitive, and the vertices of $\tG$ are ordered pairs of adjacent vertices of $\G$ (i.e., arcs), the action of $\tilde{G}$ on $\tG$ is vertex-transitive.
\end{proof}

Not all vertex-transitive generalized truncations arise from the construction of Theorem~\ref{the:VTfromAT}; we give an example of such a vertex-transitive generalized truncation in the second part of this section. However, in order to find such truncations and to ensure vertex-transitivity, one has to allow for mixers. Moreover, even in the situation of Theorem~\ref{the:VTfromAT}, one can have automorphisms of $\G$ that do not lift to $\Aut(\tG)$ and we can have automorphisms of $\tG$ that do not project to $\Aut(\G)$. In the remainder of this section we give examples of both situations. In view of Corollary~\ref{cor:small_girth}, the girth of the graph $\U = (\G(v),\O_v)$ must be at least twice the girth of $\G$ for the latter possibility to occur. 

\begin{example}
Recall that the truncation $\tG = T(K_5, \rho_1; C_4)$ from Example~\ref{ex:K_5_trunc_C_4} is vertex-transitive. We now show that it can be constructed using Theorem~\ref{the:VTfromAT}. 
If we take the group of automorphisms $G = \langle (a\,b\,c\,e\,d), (b\,c\,d\,e) \rangle$, the set of $2$-sets \[ \O_a = \{\{b,c\},\{c,d\},\{d,e\},\{e,b\}\} \]
consists of
a single orbit of the action of $G_a = \langle (b\,c\,d\,e)\rangle$ on the $2$-sets of elements from $\{b,c,d,e\}$. Since $G$ is a maximal subgroup in $\mathrm{Sym}_5$, and $\O_a$ is clearly not a union of orbits of the action of the stabilizer of $a$ in $\mathrm{Sym}_5$ on the $2$-sets of elements 
from $\{b,c,d,e\}$, Corollary~\ref{cor:small_girth} and Corollary~\ref{cor:liftsVT} imply that the automorphism group of the obtained truncation is actually the lift of $G$, and is thus isomorphic to $G$. Thus,  $\tG$ is an example of a vertex-transitive generalized truncation obtained via Theorem~\ref{the:VTfromAT} for which the maximal subgroup of the truncated graph that lifts is a proper subgroup of its automorphism group. Observe that since $G$ is (up to conjugacy) the only arc-transitive subgroup of $\Aut(K_5)$ which is not $3$-transitive (otherwise the stabilizer of a point has just one orbit on the $2$-sets and we will not get a cycle $C_4$ for the inserted graph $\U$), this is the only vertex-transitive generalized truncation of $K_5$ by $C_4$ that arises in the context of Theorem~\ref{the:VTfromAT}.
\end{example}

\begin{example}
Let us next consider all possible vertex-transitive generalized truncations of the complete graph $K_6$ by $C_5$ which arise in the context of Theorem~\ref{the:VTfromAT}. As argued above, all of them arise from
an arc-transitive subgroup of $\Aut(K_6) = \S_6$ which is not $3$-transitive. 
Only one conjugacy class of such subgroups exists, namely the class consisting 
of subgroups isomorphic to the alternating group $\mathrm{Alt}_5$. The corresponding action of a vertex stabilizer on the $2$-sets affords two orbits giving rise to $C_5$. Both choices result in the same generalized truncation (up to isomorphism), given in Figure~\ref{fig:K_6trunc}.
\begin{figure}[!h]
\begin{center}
\includegraphics[scale=0.45]{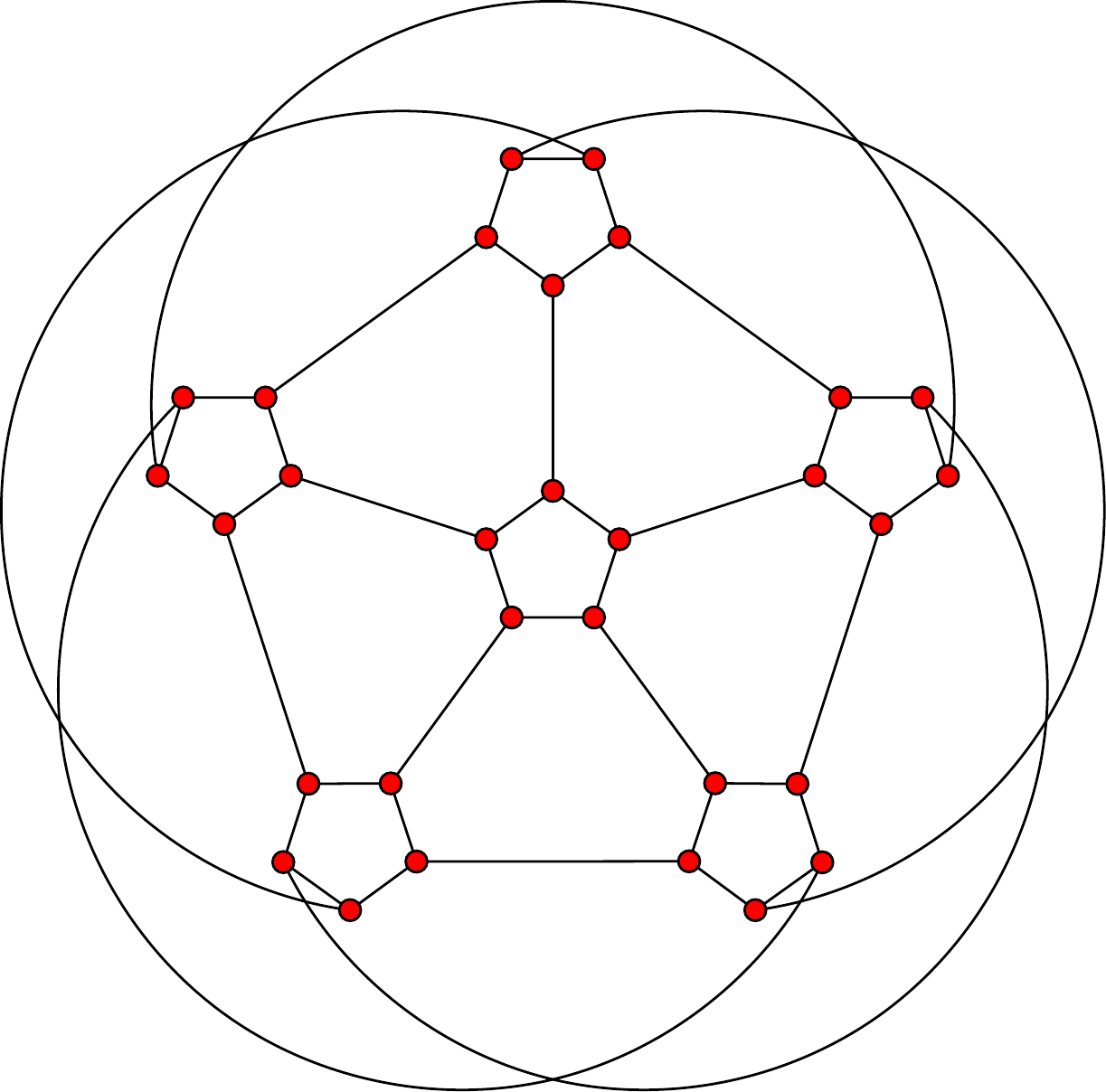}
\end{center}
\caption{The unique vertex-transitive generalized truncation of $K_6$ by $C_5$.}
\label{fig:K_6trunc}
\end{figure}

Since the length of the inserted cycles is $5$, Corollary~\ref{cor:small_girth} and Corollary~\ref{cor:liftsVT} imply that the automorphism group of the obtained generalized truncation is in fact isomorphic to $\mathrm{Alt}_5$. We remark that the generalized truncation from Figure~\ref{fig:K_6trunc} is not isomorphic to the one in~\cite[p. 2614]{ExoJaj12}, since that graph is not vertex-transitive. \hfill $\blacktriangle$
\end{example}

The previous two examples may appear to suggest that for each $n \geq 4$ there exists precisely one (vertex-transitive) generalized truncation of $K_n$ by $C_{n-1}$ that arises in the context of Theorem~\ref{the:VTfromAT}. We show in the remainder of this section that this is far from being true. Of course, the existence of any such generalized truncation requires the existence of a suitable $2$-transitive subgroup $G$ of $\S_n$. Since $G$ would have to be $2$-transitive, the lengths of the orbits of the action of the stabilizer $G_1$ on the $2$-sets from $\{2,3,\ldots , n\}$ would 
be multiples of $n-1$, or multiples of $(n-1)/2$, for odd $n$. But for the inserted graph $\U$ to be isomorphic to $C_{n-1}$, the set $\O_1$ would have to contain precisely two pairs containing $j$ for each $j \in \{2,3,\ldots , n\}$, and so $\O_1$ would have to consist of a single orbit of length $n-1$ or of two orbits of length $(n-1)/2$. In particular, $G$ would have to be of order $n(n-1)$ or $2n(n-1)$. 

This observation allows us to investigate the situation for small values of $n$ using 
{\sc Magma}~\cite{Mag}. The results of our investigation are presented in Table~\ref{tab:K_n}, which lists the parameter $n$ followed by the order $n(n-1)$ of $\tG$, the girth of $\tG$, the order $|\Aut(\tG)|$, and an indication of whether $\Aut(\tG)$ coincides with the lift of $G$. The table includes each generalized truncation $\tG = T(K_n,G;\O_1)$ (up to isomorphism) of a complete graph $K_n$ by a cycle $C_{n-1}$, arising in the context of Theorem~\ref{the:VTfromAT}, up to $n = 20$, 
with the exception of $n = 16$.

\begin{table}
{\small
$$
\begin{array}{c|c|c|c|c}
n  &   \mathrm{order} & \mathrm{girth}(\tG)  &  |\Aut(\tG)|  &  \Aut(\tG) = \tilde{G}\\ \hline
4  &  12  & 3  &  24  &  \mathrm{true}\\
5  &  20  & 4  &  20  &  \mathrm{true}\\
6  &  30  & 5  &  60  &  \mathrm{true}\\
7  &  42  & 6  &  126  &  \mathrm{false}\\
8  &  56  & 7  &  56  &  \mathrm{true}\\
9  &  72  & 8  &  72  &  \mathrm{true}\\
11  &  110  & 10  &  110  &  \mathrm{true}\\
11  &  110  & 10  &  1320  &  \mathrm{false}\\
13  &  156  & 9  &  156  &  \mathrm{true}\\
13  &  156  & 9  &  156  &  \mathrm{true}\\
17  &  272  & 11  &  272  &  \mathrm{true}\\
17  &  272  & 11  &  272  &  \mathrm{true}\\
17  &  272  & 12  &  272  &  \mathrm{true}\\
17  &  272  & 13  &  272  &  \mathrm{true}\\
19  &  342  & 10  &  342  &  \mathrm{true}\\
19  &  342  & 12  &  342  &  \mathrm{true}\\
19  &  342  & 12  &  342  &  \mathrm{true}
\end{array}
$$
\caption{Information on the generalized truncations of $K_n$ by $C_{n-1}$ arising from Theorem~\ref{the:VTfromAT}, where $n \leq 20$, $n \neq 16$.}
\label{tab:K_n}
}
\end{table}

The obtained information reveals that for some $n$ (e.g., $10, 12, 14, 15, 18$, and $20$) there is no generalized truncation of a complete graph $K_n$ by $C_{n-1}$ arising in the context of Theorem~\ref{the:VTfromAT}. On the other hand, there 
exist values of $n$ for which several such generalized truncations exist. Moreover, they 
may be of different girths, and two such generalized truncations may differ in that 
one of them admits mixers while the other does not (see the case $n = 11$). 
It is also interesting to note that one of the truncations of $K_{17}$, all of which
are of order $272$, achieves the girth $13$. The order $272$ is the smallest order for 
which a $3$-regular graph of girth $13$ is known to exist \cite{ExoJaj08}.

We therefore propose the following problem.

\begin{problem}
\label{pro:first}
For each natural number $n$, classify the generalized truncations of the complete graph $K_n$ by the cycle $C_{n-1}$ that arise in the context of Theorem~\ref{the:VTfromAT}.
\end{problem}

%In particular, since the smallest known $3$-regular graph of girth $15$ is of order 
%$620$ \cite{ExoJaj08}, finding a $3$-regular truncation of $K_n$ of girth $15$, 
%for $ 20 \leq n \leq 25 $, would yield a new smallest known $3$-regular graph of 
%girth $15$.

Since all generalized truncations of complete graphs $K_n$ by cycles $C_{n-1}$ 
%arising in the context of Theorem~\ref{the:VTfromAT}, 
result in cubic graphs, it might prove useful to know 
whether all vertex-transitive generalized truncations of a complete graph $K_n$ by a cycle $C_{n-1}$ arise in the context of Theorem~\ref{the:VTfromAT}.

If $n \leq 6$, the inserted cycle $C_{n-1}$ is of girth less than $6$, and so Corollary~\ref{cor:small_girth} implies that in this case the whole automorphism group of the generalized truncation projects (and is clearly arc-transitive on $K_n$). 
Therefore, none of these
cases can result in a vertex-transitive generalized truncation of $K_n$ by $C_{n-1}$ which does not arise in the context of Theorem~\ref{the:VTfromAT}. The first three lines of Table~\ref{tab:K_n} thus correspond to the only three vertex-transitive generalized truncations of $K_n$ by $C_{n-1}$ for $n$ up to $6$. 

For $n \geq 7$, this approach does not work anymore. Nevertheless, since a vertex-transitive generalized truncation of a complete graph $K_n$ by a cycle $C_{n-1}$ is a vertex-transitive cubic graph of order $n(n-1)$, one can inspect the census of cubic vertex-transitive graphs of order up to 1280 due to Poto\v cnik, Spiga and Verret~\cite{PotSpiVer13}, and search for examples which arise as such truncations. For example, relying on
{\sc Magma}, it can be verified that there are $6$ connected cubic vertex-transitive graphs of order $42$ and girth $6$, but only one of them is a generalized truncation of $K_7$ by $C_6$ (the graph is given in Figure~\ref{fig:K_7trunc}). This graph must therefore
correspond to row 4 of Table~\ref{tab:K_n}, and arise via Theorem~\ref{the:VTfromAT}. The maximal subgroup of $\S_7$ that lifts is of order $42$ (and so the obtained graph is in fact a Cayley graph), while the full automorphism group of the truncation is of order $126$. This implies that this generalized truncation allows mixers. 

\begin{figure}[!h]
\begin{center}
\includegraphics[scale=0.4]{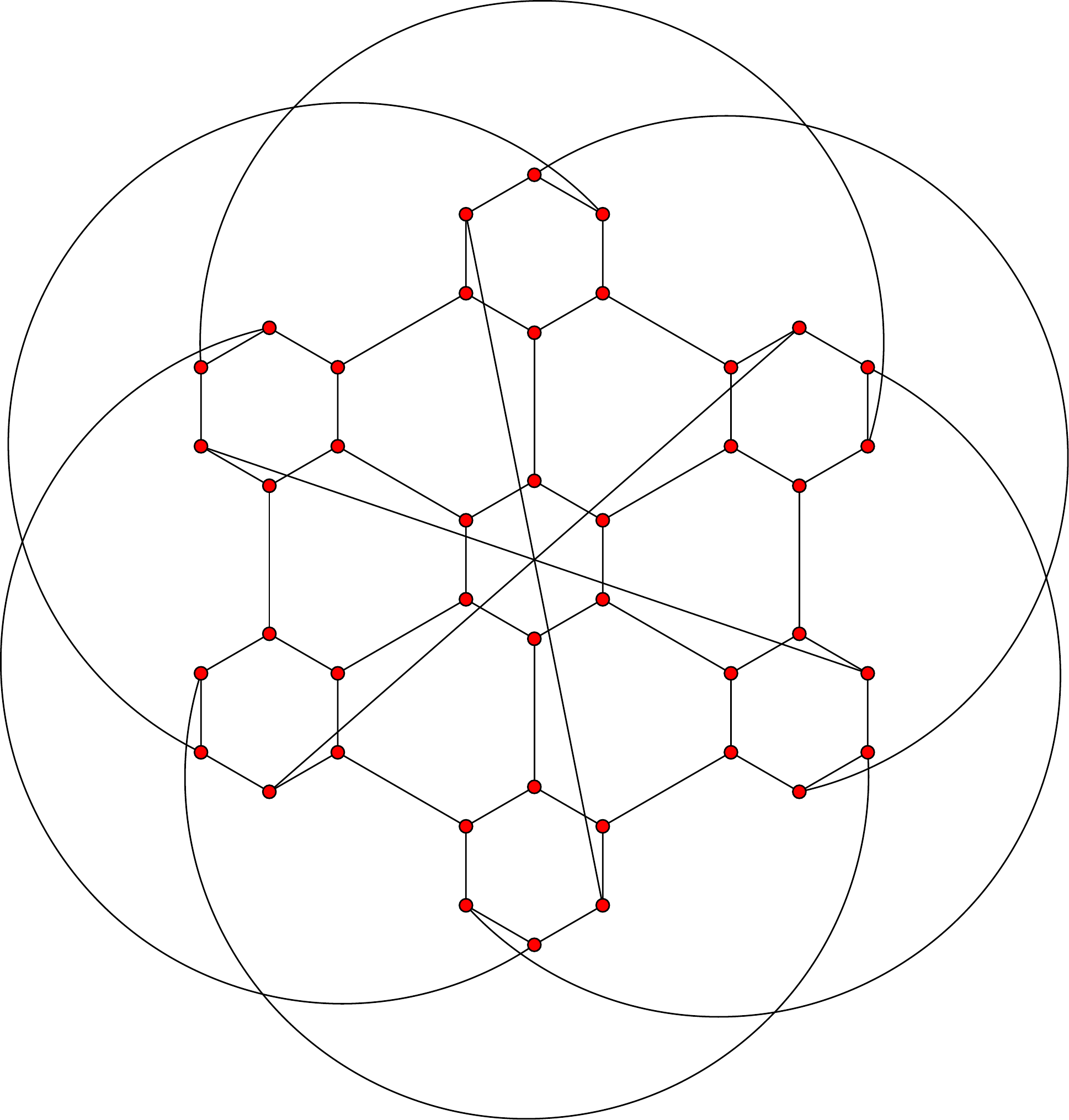}
\end{center}
\caption{The unique vertex-transitive generalized truncation of $K_7$ by $C_6$.}
\label{fig:K_7trunc}
\end{figure}

For $n = 8$, we find that there are two connected vertex-transitive cubic graphs of order $56$ possessing $7$-cycles (both of which are in fact of girth $7$). Only one of them is a generalized truncation of $K_8$ by $C_7$, which
therefore has to correspond to row 5 of Table~\ref{tab:K_n} and arise via Theorem~\ref{the:VTfromAT}. The situation changes for $n = 9$. Namely, there exist two connected cubic vertex-transitive graphs of order $72$ which are generalized truncations of $K_9$ by $C_8$. One of them corresponds to row~6 of Table~\ref{tab:K_n}, while the other one does not arise via Theorem~\ref{the:VTfromAT}. It is the Cayley graph of the group 
$$
G = \langle \; a,b,c \mid a^2, b^2, c^2, acabcbcb, abcacbcacb, (ac)^6 \; \rangle
$$ 
with respect to the connection set $S = \{a,b,c\}$. The graph contains 
nine pairwise disjoint $8$-cycles corresponding to the relation $acabcbcb = 1$, which represent the nine $8$-cycles inserted into $K_9$ to construct the graph as a generalized truncation of $K_9$ by $C_8$. Since this graph does not arise in the context of Theorem~\ref{the:VTfromAT}, it follows that its automorphism group (which coincides with $G$) does not project. This example shows that there exist vertex-transitive truncations of complete graphs by cycles which do not arise via Theorem~\ref{the:VTfromAT}. We therefore generalize Problem~\ref{pro:first} to a wider setting.

\begin{problem}
For each natural number $n$ classify all vertex-transitive generalized truncations of the complete graph $K_n$ by the cycle $C_{n-1}$.
\end{problem}

We finish this section with a useful result, giving a sufficient condition for a vertex-transitive graph to be a generalized truncation obtained via Theorem~\ref{the:VTfromAT}. This result is used in Section~\ref{sec:cubVT} to obtain a characterization of all cubic vertex-transitive graphs of girth at most $5$.

\begin{theorem}
\label{the:VTtrunc}
Let $\Gamma$ be a vertex-transitive graph possessing 
a vertex-transitive group $G$ of automorphisms admitting a nontrivial imprimitivity block system $\mathcal {B}$ on $V(\Gamma)$. If 
there exists a block $B \in \cB$ with the property that each vertex of $B$ has exactly one neighbor outside $B$ and no two vertices of $B$ have a neighbor in the same $B' \in \cB$, $B' \neq B$, then $\Gamma$ is a generalized truncation of an arc-transitive graph by a vertex-transitive graph in the sense of Theorem~\ref{the:VTfromAT}. 
\end{theorem}

\begin{proof}
Let $\Gamma_\cB$ be the quotient graph with respect to $\cB$, that is, the graph with vertex set $\cB$ and the adjacency of vertices $B,B' \in \cB$ determined by the existence of a pair of vertices $v \in B$ and $v' \in B'$
adjacent in $\Gamma$. Since every vertex of $B$ has precisely one neighbor outside $B$ and no two vertices of $B$ have a neighbor in the same $B' \neq B$, the valence of $B$ in $\Gamma_\cB$ is $|B|$, and since $\Gamma_\cB$
is vertex-transitive, it is $|B|$-regular. Consider the natural induced action of $G$ on $\Gamma_\cB$. Using the same argument as in the proof of Lemma~\ref{le:kernel}, we see that this action is faithful, and so we can identify $G$ with a subgroup of $\Aut(\G_\cB)$.  While the action of $G$ on $\cB$ is necessarily vertex-transitive, we can show that it is 
also arc-transitive. To argue this, let $(B_1,B_2)$ be an arbitrary arc of $\Gamma_\cB$.  Due to our assumptions, there exists a unique pair of vertices $u_1 \in B_1$ and $u_2 \in B_2$ such that $u_1 \sim u_2$ in $\Gamma$. Let $B_3$ and $B_4$ be another pair of adjacent vertices in $\Gamma_\cB$, and let $u_3 \in B_3$, $u_4 \in B_4$, be their unique pair of 
vertices adjacent in $\Gamma$. Since $G$ is transitive on $\Gamma$, there exists $g \in G$ mapping $u_1$ to $u_3$. Consequently, the induced action of $g$ maps $B_1$ to $B_3$. Since $u_2$ is the unique neighbor of $u_1$ outside $B_1$ and $u_4$ is the unique neighbor of $u_3$ outside $B_3$, it follows that $g$ maps $u_2$ to $u_4$, and therefore also $B_2$ to $B_4$. It thus maps the arc $(B_1,B_2)$ to the arc $(B_3,B_4)$, proving that $G$ induces an arc-transitive action on $\Gamma_\cB$, as claimed. 

Fix an arbitrary $B \in \cB$ and denote by $\U$ the subgraph of $\Gamma$ induced on $B$. Since $G$ is vertex-transitive and $\cB$ is an imprimitivity block system for $G$, the graph $\U$ is vertex-transitive and independent of the particular choice of $B \in \cB$. By assumption, $\U$ is a $(k-1)$-regular graph of order $|B|$, where $k$ is the valence of $\Gamma$. We now prove that $\Gamma$ is a generalized truncation of $\Gamma_\cB$ with respect to a suitable union $\O_B$ of orbits of the action of $G_B$ on $2$-sets of neighbors of $B$ in $\Gamma_\cB$. Recall that each $v \in B$ uniquely determines the block $B_v \in \cB$ such that $B \sim B_v$ in $\Gamma_\cB$. Let $\O_B$ be the set of all pairs $\{B_u, B_v\}$ such that $u \sim v$ in $\Gamma$ (and thus in $\U$). Clearly, $\O_B$ is a union of orbits of the action of $G_B$ on the $2$-sets of neighbors of $B$ in $\Gamma_\cB$ and the graph corresponding to $\O_B$ is isomorphic to $\U$. It is now easy to see that $\Gamma$ is isomorphic to the generalized truncation $T(\G_\cB, G; \O_B)$.
\end{proof}

%%%%%%%%%%%%%%%%%%%%%%%%%%%%%%
%----------------  An application: a characterization of cubic vertex-transitive graphs of small girth  -----------------------------
%%%%%%%%%%%%%%%%%%%%%%%%%%%%%%
\section{Characterization of cubic vertex-transitive graphs of girths $3$, $4$, and $5$}
\label{sec:cubVT}

In this section, we apply the results obtained in the previous sections to obtain
a characterization of cubic vertex-transitive graphs of girth at most $5$. Cubic arc-transitive graphs of small girth have been extensively studied (see, for example, \cite{ConNed07, FenNed06, KutMar09, Mil71}). It is well known that there are only five connected cubic arc-transitive graphs of girth at most $5$; there is one of girth $3$ (the complete graph $K_4$), two of girth $4$ (the complete bipartite graph $K_{3,3}$ and the cube $Q_3$), and two of girth $5$ (the Petersen graph and the Dodecahedron graph). The family of connected vertex-transitive cubic graphs is however much richer and it contains infinitely many examples of each of 
the girths $3$, $4$ and $5$ (see the remarks in the following three subsections). It is the aim of this section to find a characterization of these graphs. 

In view of the above remarks, we only need to characterize those vertex-transitive graphs which are not arc-transitive. This might be done along the lines of~\cite{PotSpiVer13} via dividing
the characterization into two cases, one in which the stabilizer of a vertex in the automorphism group of such a graph is trivial, and one in which it has two orbits on the neighbors of the vertex, and then studying the corresponding constructions from~\cite{PotSpiVer13}. We prefer the following much more elementary approach.

%%%%%%%%%%%%%%   Girth 3
\subsection{Girth 3}
\label{subsec:girth3}

We first consider the easiest case, namely graphs of girth $3$. Recall that the {\em prism graph} $\Pr(n)$, $n \geq 3$, is the cubic Cayley graph of the group $\ZZ_2 \times \ZZ_n$ with respect to the connection set $S = \{(1,0),(0,1),(0,n-1)\}$ (these graphs are also known as generalized Petersen graphs $\mathrm{GP}(n,1)$).

\begin{theorem}
\label{the:girth3}
Let $\Gamma$ be a connected cubic graph of girth $3$. Then $\Gamma$ is vertex-transitive if and only if it is either the complete graph $K_4$, the prism $\Pr(3)$, or a generalized truncation of an arc-transitive cubic graph $\Lambda$ by the $3$-cycle $C_3$, in which case $\Aut(\Gamma) \cong \Aut(\Lambda)$.
\end{theorem}

\begin{proof}
Both $K_4$ and $\Pr(3)$ are Cayley, and therefore vertex-transitive, cubic graphs of girth $3$.
Theorem~\ref{the:VTfromAT} implies that generalized truncations of arc-transitive cubic graphs 
by $3$-cycles are vertex-transitive graphs of girth $3$ which, by Corollary~\ref{cor:small_girth}, share the automorphism group with the 
underlying arc-transitive graph.

To prove the converse, suppose that $\Gamma$ is vertex-transitive. The vertex-transitivity implies that each vertex of $\Gamma$ either lies on a unique $3$-cycle or lies on three $3$-cycles. In the latter case, $\Gamma \cong K_4$, and so we can assume that each vertex of $\G$ lies on a unique $3$-cycle. This implies that the set of $3$-cycles of $\G$ forms an imprimitivity block system for $\Aut(\G)$. Moreover, each vertex of $\Gamma$ is incident to two edges 
which are part of some $3$-cycle of $\Gamma$, and one edge that does not lie on any $3$-cycle of $\Gamma$. Let us color the edges contained in $3$-cycles red, and the other edges
black, so that all $3$-cycles consist of three red edges. Each automorphism of $\Gamma$ necessarily preserves the colors of the edges. Now, let $(v_0,v_1,v_2)$ be an arbitrary $3$-cycle of $\Gamma$. Since every vertex of $\Gamma$ is incident to exactly one black edge, 
each of $v_i$, $0 \leq i \leq 2$, is adjacent to a unique neighbor $u_i \notin \{v_0, v_1, v_2\}$. 

Suppose that the set $\{u_0, u_1, u_2\}$ is not an independent set of vertices, say $u_0 \sim u_1$. Since $\Gamma$ is vertex-transitive, there exists $g \in \Aut(\Gamma)$ mapping $v_0$ to $v_2$. Then the black edge $u_0v_0$ is mapped to the black edge $u_2v_2$, while the edge $u_1v_1$ is mapped to one of $u_0v_0$ and $u_1v_1$, implying that $u_2$ must be adjacent to $u_0$ or $u_1$. Clearly, $\Gamma \cong \Pr(3)$.

It remains to consider the possibility where $\{u_0, u_1, u_2\}$ is an independent set of vertices. Let $w,w'$ be the two neighbors of $u_0$ different from $v_0$. Since the edges $u_0w$, $u_0w'$ are red, $(u_0,w,w')$ is a $3$-cycle of $\Gamma$, and thus the edge $ww'$ is also red. Since each of the vertices $u_1, u_2$ already has its `black neighbor' $v_1, v_2$, respectively, and each of $w$ and $w'$ already has its two `red neighbors', none of $u_1$ and $u_2$ can be adjacent to any of $w$ and $w'$, and so each vertex of a $3$-cycle $C$ is adjacent to precisely one vertex outside $C$ and no two vertices of $C$ have neighbors in the same $3$-cycle $C'$, different from $C$. Therefore, Theorem~\ref{the:VTtrunc} applies. The
isomorphism $\Aut(\Gamma) \cong \Aut(\Lambda)$ follows from Corollary~\ref{cor:small_girth} and Theorem~\ref{the:VTfromAT}.
\end{proof}

\begin{remark}
It is well known that infinitely many cubic arc-transitive graphs exist. This yields the existence of infinitely many cubic vertex-transitive graphs of girth $3$.
\end{remark}

%%%%%%%%%%%%%%   Girth 4
\subsection{Girth 4}
\label{subsec:girth4}

Next, we focus on cubic vertex-transitive graphs of girth $4$. We first introduce two families of such graphs. 

The first family consists of the well known {\em M\"{o}bius ladders}. The graph $\Ml(n)$, where $n \geq 3$, is the Cayley graph $\Cay(\ZZ_{2n};\{\pm 1, n\})$. It is clear that $\Ml(n)$ is a cubic vertex-transitive graph of girth $4$, for all $ n \geq 3 $. 

The members of the second family are called generalized prisms. For any integer $n \geq 2$, the {\em generalized prism} $\GPr(n)$ is the graph with vertex set $V = \{(i,j)\colon i \in \ZZ_2, j \in \ZZ_{2n}\}$ in which each vertex $(i,j)$ is adjacent to $(i,j+1)$ and in addition $(i,j)$ is adjacent to $(i+1,j+1)$ for all even $j$ (the computations are done modulo $2$ on the first component, and modulo $2n$ on the second component). 
It is not hard to see that generalized prisms are Cayley graphs 
of girth $4$; $\GPr(n) \cong \Cay(\mathrm{D}_{2n};S)$, 
$\mathrm{D}_{2n} = \langle t,r \mid t^2, r^{2n}, (tr)^2 \rangle$,
$S = \{t, tr, tr^n\}$.

The M\"{o}bius ladder $\Ml(5)$ and the generalized prism $\GPr(5)$ are depicted in Figure~\ref{fig:twistPrism}. 
\begin{figure}[!h]
\begin{center}
\includegraphics[scale=0.6]{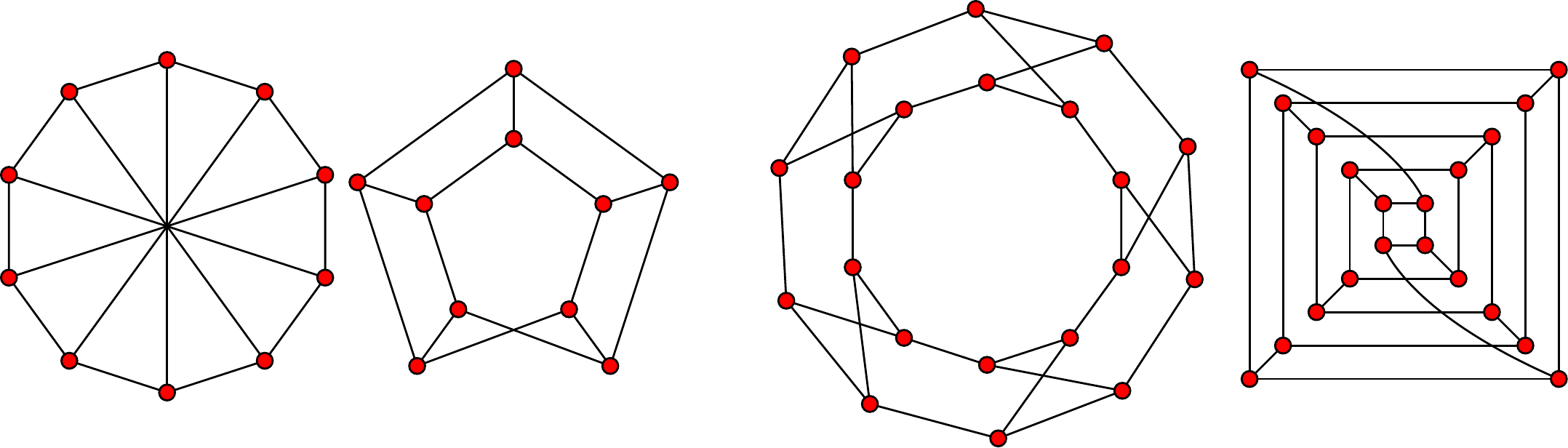}
\end{center}
\caption{Two drawings of each of $\Ml(5)$ and $\GPr(5)$.}
\label{fig:twistPrism}
\end{figure}

%\begin{lemma}
%\label{le:genPrism}
%Let $n \geq 2$ be an integer. The generalized prism $\GPr(n)$ 
%is isomorphic to the Cayley graph of the dihedral group $\mathrm{D}_{2n} = \langle t,r \mid t^2, r^{2n}, (tr)^2 \rangle$ with respect to the connection set $S = \{t, tr, tr^n\}$. Its girth is equal to $4$.
%\end{lemma}

%\begin{proof}
%That for any $n \geq 2$ the graph $\GPr(n)$ is cubic and of girth $4$ is clear from the definition. To prove that it is vertex-transitive (which is in fact straightforward to verify directly from the definition) it thus suffices to prove that it is isomorphic to the Cayley graph $\Cay(D_{2n} ; \{t, tr, tr^n\})$. This however, is also easy to see. For each even $j = 2j_0$ map the vertices $(0,j), (1,j), (0,j+1), (1,j+1)$ to $r^{j_0}, r^{n+j_0}, tr^{-j_0}, tr^{n-j_0}$, respectively. It is easy to see that the obtained mapping is an isomorphism of graphs.  
%\end{proof}

We now have the ingredients needed for a characterization of cubic vertex-transitive graphs of girth $4$.

\begin{theorem}
\label{the:girth4}
Let $\Gamma$ be a connected cubic graph of girth $4$ and order $2n$. Then $\Gamma$ is vertex-transitive if and only if it is isomorphic to the prism $\Pr(n)$ with $n \geq 4$, the M\"{o}bius ladder $\Ml(n)$ with $n \geq 3$, the generalized prism $\GPr(\frac{n}{2})$ ($n$ even), or it is isomorphic to a generalized truncation of an arc-transitive tetravalent graph $\Lambda$ by the $4$-cycle $C_4$ in the sense of Theorem~\ref{the:VTfromAT}, in which case $\Aut(\Gamma) \cong \Aut(\Lambda)$.
\end{theorem}

\begin{proof}
Vertex-transitivity and girth $4$ of the graphs listed in the theorem follow from the preceding remarks
%, Lemma~\ref{le:genPrism} 
and Theorem~\ref{the:VTfromAT}. 

To prove the converse, suppose that $\Gamma$ is vertex-transitive of girth $4$. 
Let $C = (v_0,v_1,v_2,v_3)$ be an arbitrary $4$-cycle of $\Gamma$. Suppose first that there exists a pair of vertices of $C$ having a common neighbor outside $C$. Since $\Gamma$ is of girth $4$, this can only hold for the pairs $v_0,v_2$ or $v_1, v_3$. Without loss of generality, we may assume that $v_0$ and $v_2$ have a common neighbor $w$ outside $C$. This means, in particular, that all three $2$-paths containing $v_0$ at their center belong to some $4$-cycle. The vertex-transitivity of $\G$ implies the same for each vertex of $\Gamma$, i.e., 
for each vertex $v$ of $\Gamma$ and any two of its neighbors $u,u'$, the $2$-path $(u,v,u')$ lies on a $4$-cycle of $\Gamma$. Let now $w' \neq w$ (recall that $\G$ is of girth $4$) be the neighbor of $v_1$ outside $C$. As argued above, both $(w',v_1,v_0)$ and $(w',v_1,v_2)$ lie on a $4$-cycle of $\Gamma$, and so $w'$ is adjacent to at least one of $v_3$ and $w$. It is easy to see that it has to be adjacent to both of them, implying that $\Gamma \cong K_{3,3} \cong \Ml(3)$.

For the rest of the proof we can thus assume that no two vertices contained in a $4$-cycle of $\Gamma$ share a neighbor outside this $4$-cycle. Let $C = (v_0,v_1,v_2,v_3)$ be as above, and for each $0 \leq i \leq 3$, let $u_i$ be the neighbor of $v_i$ outside $C$. We proceed by considering two cases.
\medskip

\noindent
{\bf Case 1:} {\em There is an $i$ such that $u_i$ is adjacent to $u_{i+1}$}
(indices computed modulo $4$). \\
Without loss of generality we may assume $u_0 \sim u_1$. If also $u_1 \sim u_2$ (or $u_0 \sim u_3$), then, once again, for any vertex $v$ of $\Gamma$ and any two of its neighbors $u,u'$, there is a $4$-cycle containing the $2$-path $(u,v,u')$.
It is clear that in this case $\Gamma \cong \Pr(4)$; the underlying graph of the cube.
If, on the other hand, we assume that $u_0 \sim u_1$, but none of $u_0 \sim u_3$ and $u_1 \sim u_2$ holds, the edges $v_1v_2$ and $v_1u_1$ each belong to a unique $4$-cycle of $\Gamma$, while $v_1v_0$ belongs to two. Since $\Gamma$ is vertex-transitive, each vertex of $\Gamma$ must be incident with one edge contained in two $4$-cycles of $\Gamma$ and with two edges each contained in just
one $4$-cycle of $\Gamma$. This divides the edges of $\Gamma$ into two
disjoint sets. Let us color the edges of $\Gamma$ lying within two $4$-cycles red and the other edges black (thus $v_0v_1$ is red, and $v_0v_3$ and $v_1v_2$ are black). Observe that no $4$-cycle contains two consecutive black edges as in such a case there would be no $4$-cycle through the remaining edge incident to the common endpoint of these two black edges. Thus $v_2v_3$ is red while $u_iv_i$ is black for all $0 \leq i \leq 3$. It follows that $u_2 \sim u_3$ and $u_2u_3$ and $u_0u_1$ are both red. It is now clear that $\Gamma$ is isomorphic either to the prism $\Pr(n)$ (in which case $n \geq 4$, for the girth to be $4$) or to the M\"{o}bius ladder $\Ml(n)$ with $n \geq 4$.
\medskip

\noindent
{\bf Case 2:} {\em For each $i$,  $u_i$ is adjacent to neither $u_{i-1}$ nor $u_{i+1}$.}
\\
It follows that $v_0$ (and hence any vertex of $\Gamma$) lies on a unique $4$-cycle of $\Gamma$. Each vertex of $\G$ is thus incident to two edges contained in
a unique $4$-cycle (we color such edges red), and to one edge that does not lie on any $4$-cycle of $\Gamma$ (we color such edges black). Thus, the red cycles 
form a
complete $2$-factor, the black edges form a complete $1$-factor of $\Gamma$, and both factors are preserved by the  automorphisms of $\Gamma$. Now, if $u_0$ were adjacent to $u_2$, then, since $v_0u_0$ is black, $u_0u_2$ would be red. Moreover, starting at $u_0$, traversing the black edge to $v_0$, taking (any) two consecutive red edges (contained in $C$), and then taking a black edge would 
bring us to the neighbor $u_2$ of  $u_0$ and thus give rise to a $5$-cycle. However, starting at $v_0$ and following the same sequence of colors of edges would imply that some red neighbor (a neighbor along a red edge) of $u_2$ would have to be adjacent to one of $v_1$ and $v_3$ via a black edge, forcing $u_2$ to be a neighbor of one of $u_1$ and $u_3$. This contradiction thus implies that $\{u_1, u_2, u_3, u_4\}$ is an independent set.

Let $w, w'$ be the two red neighbors of $u_0$ (see the left part of Figure~\ref{fig:girth4proof}) and let $z$ be their common red neighbor, different from $u_0$. Suppose first that $z \in \{u_1, u_3\}$; without loss of generality assume $z = u_1$. Letting $g \in \Aut(\Gamma)$ be such that it maps $u_0$ to $v_0$ we see that $u_1$ is mapped to $v_2$ (since $u_1$ is the antipodal vertex to $u_0$ on the unique $4$-cycle containing $u_0$ while $v_2$ is the antipodal vertex of $v_0$ on the unique $4$-cycle containing $v_0$). Then the black neighbor $v_0$ of $u_0$ is mapped by $g$ to $u_0$ and the black neighbor $v_1$ of $u_1$ is mapped to $u_2$. But since $v_0 \sim v_1$, it follows that $u_0 \sim u_2$, a contradiction that proves that $z \notin \{u_1, u_3\}$.
\begin{figure}[h]
\begin{center}
\includegraphics[scale=0.7]{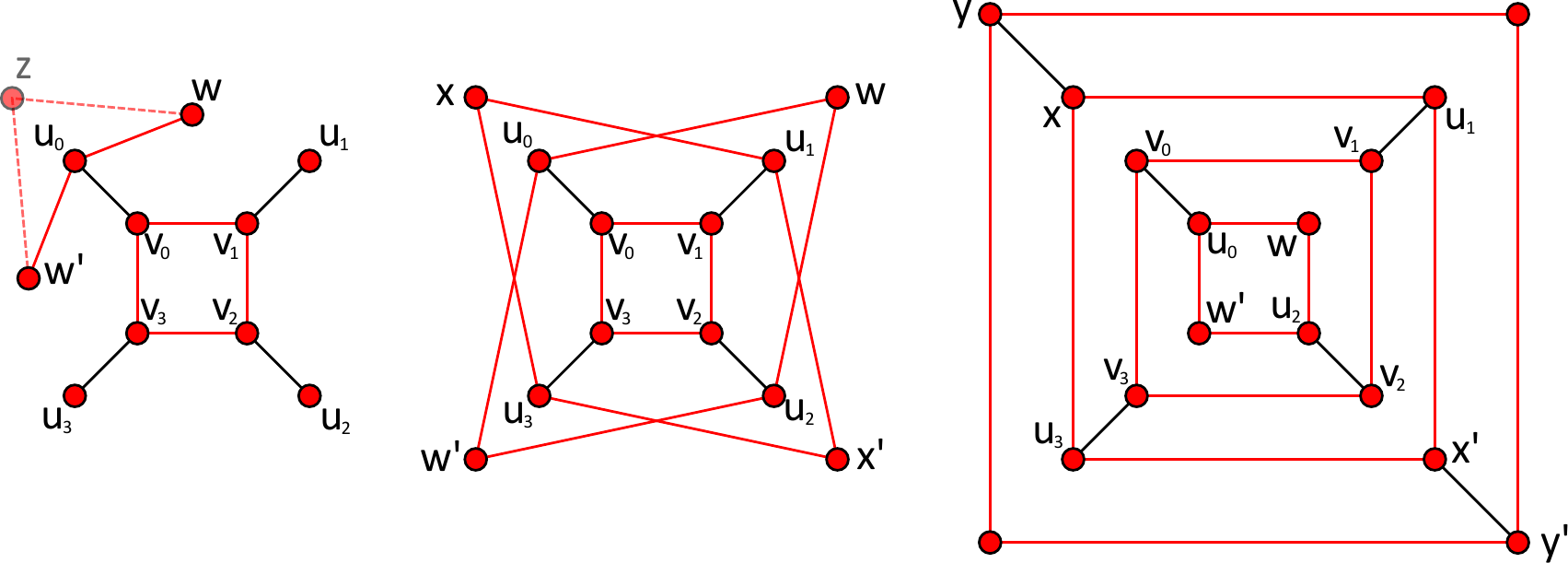}
\end{center}
\caption{The local situations around a $4$-cycle in the case that each vertex is on a unique $4$-cycle.}
\label{fig:girth4proof}
\end{figure}

Suppose next that $z = u_2$. Then $u_0$ and $u_2$, which are black neighbors of the antipodal pair $v_0, v_2$, are also antipodal on a $4$-cycle. Mapping $v_0$ to $v_1$ we thus see that $u_1$ and $u_3$ also must be antipodal on a (red) $4$-cycle. Their two common red neighbors $x$ and $x'$ are of course different from $w$ and $w'$ since both of $w$ and $w'$ already have their two red incident edges (see the middle part of Figure~\ref{fig:girth4proof}). Now, if at least one of $x$ and $x'$ is adjacent to one of $w$ and $w'$, say $x \sim w$, then the fact that antipodal pairs are connected with black edges to antipodal pairs implies that also $x' \sim w'$ must hold, and we get $\Gamma \cong \GPr(3)$. Otherwise, letting $y$ and $y'$ be the black neighbors of $x$ and $x'$, respectively, we see that they have a pair of common red neighbors different from $w$ and $w'$ (see the right part of Figure~\ref{fig:girth4proof}). Continuing this way, we easily conclude that $\Gamma \cong \GPr(\frac{n}{2})$, $n$ even. 

We are finally left with the possibility that the common black neighbor $z$ of $w$ and $w'$, different from $u_0$, is none of $u_1, u_2$ or $u_3$. But then each vertex $v_i$ of the $4$-cycle $C$ is adjacent to precisely one vertex outside $C$ and no two vertices of $C$ have a neighbor in the same $4$-cycle $C'$, different from $C$. Since the $4$-cycles of $\Gamma$ clearly form an imprimitivity block system for the automorphism group $\Aut(\Gamma)$, we can apply Theorem~\ref{the:VTtrunc}. The final claim $\Aut(\Gamma) \cong \Aut(\Lambda)$ follows from Corollary~\ref{cor:small_girth} and Theorem~\ref{the:VTfromAT}.
\end{proof}

\begin{remark}
It is well known that there exist infinitely many tetravalent arc-transitive graphs (see for instance~\cite{Wil08}), and so Theorem~\ref{the:girth4} implies that there exist infinitely many cubic vertex-transitive graphs of girth $4$, not isomorphic to a generalized prism, having the property that each vertex lies on a unique $4$-cycle.
\end{remark}

%%%%%%%%%%%%%%   Girth 5
\subsection{Girth 5}
\label{subsec:girth5}

We conclude with a characterization of cubic vertex-transitive graphs of girth $5$. In what follows, we take advantage of the obvious fact that in a graph of girth $5$ no $3$-path can be contained in more than one $5$-cycle. In case of cubic graphs, 
this observation implies that any $2$-path is contained in at most two $5$-cycles. 

\begin{theorem}
\label{the:girth5}
Let $\Gamma$ be a connected cubic graph of girth $5$. Then $\Gamma$ is vertex-transitive if and only if it is either isomorphic to the Petersen graph or the Dodecahedron graph, or it is isomorphic to a generalized truncation of an arc-transitive $5$-valent graph $\Lambda$ by the $5$-cycle $C_5$ in the sense of Theorem~\ref{the:VTfromAT}. In the latter case, $\Aut(\Gamma) \cong \Aut(\Lambda)$.
\end{theorem}

\begin{proof}
The Petersen graph and the Dodecahedron graph are vertex-transitive cubic graphs of girth $5$. All other graphs from the theorem are vertex-transitive cubic graphs of girth $5$ because of Theorem~\ref{the:VTfromAT}. 

To prove the converse, let $C = (v_0, v_1, v_2, v_3, v_4)$ be a $5$-cycle of $\Gamma$. Since $\Gamma$ is cubic and of girth $5$, each of the vertices $v_i$ has its own unique neighbor $u_i$ outside $C$. We distinguish two cases depending on whether the set $\{u_1, u_2, u_3, u_4, u_5 \}$ is independent or not. 
\medskip

\noindent
{\bf Case 1:} {\em There is at least one edge connecting two of the vertices $u_i$.}\\
We show that in this case $\Gamma$ is the Petersen graph. To this end, suppose $u_0$ is adjacent to at least one of $u_2$ and $u_3$, say to $u_2$ (since $\Gamma$ is of girth $5$, $u_0$ cannot be adjacent to $u_1$ or $u_4$). This implies that each edge of $\Gamma$ lies on at least one $5$-cycle (since this holds for all three edges incident to $v_0$). Moreover, the $2$-path $(v_0,v_1,v_2)$ lies on exactly two $5$-cycles of $\Gamma$, and so each vertex of $\Gamma$ is the internal vertex of at least one $2$-path which is contained on two $5$-cycles. We claim that this implies that each $2$-path of $\Gamma$ lies on at least one $5$-cycle. If this were
not the case, then we would lose no generality by assuming that $(u_1,v_1,v_0)$ lies on no $5$-cycle of $\Gamma$. By vertex-transitivity, each vertex of $\Gamma$ is an internal vertex of at least one $2$-path not contained on any $5$-cycle. Let $w$ and $w'$ be the two neighbors of $u_1$ different from $v_1$. Since $(u_1,v_1,v_0)$ is not contained on any $5$-cycle, the only possible $2$-path with $u_1$ as its internal vertex which can lie on two $5$-cycles is $(w,u_1,w')$. Since there must also be a $2$-path with internal vertex $u_1$ that lies on no $5$-cycle, this holds for one of the $2$-paths $(w,u_1,v_1)$ and $(w',u_1,v_1)$. Without loss of generality assume that it is $(w,u_1,v_1)$. Let now $x$ and $x'$ be the two neighbors of $w$ different from $u_1$. As before, $(x,w,x')$ must be on two $5$-cycles, and without loss of generality, $(x,w,u_1)$ lies on no $5$-cycle. But then both $5$-cycles containing $(w',u_1,w)$ contain $(w',u_1,w,x')$, a contradiction. This proves our claim that each $2$-path of $\Gamma$ lies on at least one $5$-cycle. 

Now, if each $2$-path of $\Gamma$ lies on two $5$-cycles then $\Gamma$ is the Petersen graph (since each of the paths $(v_i,v_{i+1},v_{i+2})$ lies on the $5$-cycle $C$ the other $5$-cycle containing this $2$-path must be $(u_i,v_i,v_{i+1},v_{i+2},u_{i+2})$, and so each $u_i$ as adjacent to $u_{i+2}$ and $u_{i-2}$). We are thus left with the possibility that either each vertex is the internal vertex of one $2$-path lying on one $5$-cycle and of two $2$-paths lying on two $5$-cycles, or each vertex is the internal vertex of one $2$-path lying on two $5$-cycles and of two $2$-paths lying on one $5$-cycle. 

We first show that the first possibility cannot occur. Suppose to the contrary that each vertex of $\G$ is the internal vertex of one $2$-path lying on one $5$-cycle and of two $2$-paths lying on two $5$-cycles. Then each vertex of $\Gamma$ lies on five $5$-cycles and is incident to one edge lying on four $5$-cycles (we color such edges red) and to two edges each lying on three $5$-cycles (we color such edges black). Note that no $5$-cycle contains three consecutive black edges since if $(w_1,w_2,w_3,w_4)$ was on a $5$-cycle with $w_1w_2$, $w_2w_3$ and $w_3w_4$ all black, then each of $(w_1,w_2,w_3)$ and $(w_2,w_3,w_4)$ would be on just one $5$-cycle. But then $(w,w_2,w_3)$ cannot be on two $5$-cycles, where $w$ is the red neighbor of $w_2$. Since no two red edges are incident this implies that each $5$-cycle of $\Gamma$ consist of two red edges and three black edges. Now, let $C' = (w_0,w_1,w_2,w_3,w_4)$ be a $5$-cycle and with no loss of generality assume that the edges $w_0w_4$, $w_0w_1$ and $w_2w_3$ are black while $w_1w_2$ and $w_3w_4$ are red (see the left part of Figure~\ref{fig:girth5proof}). Let $x_0$ be the red neighbor of $w_0$ and let $x_2$ and $x_3$ be the black neighbors of $w_2$ and $w_3$, different from $w_3$ and $w_2$, respectively. The $2$-path $(x_0,w_0,w_1)$ lies on two $5$-cycles of $\Gamma$, one of which must contain $(x_0,w_0,w_1,w_2)$. Since the successor of $w_2$ on this $5$-cycle cannot be $w_3$ (as $(w_0,w_1,w_2,w_3)$ already lies on $C'$) we see that $x_0 \sim x_2$. Similarly $x_0 \sim x_3$. But since $x_0x_2$ and $x_0x_3$ are both black this gives a black $5$-cycle $(x_0,x_2,w_2,w_3,x_3)$, a contradiction.
\begin{figure}[h]
\begin{center}
\includegraphics[scale=0.7]{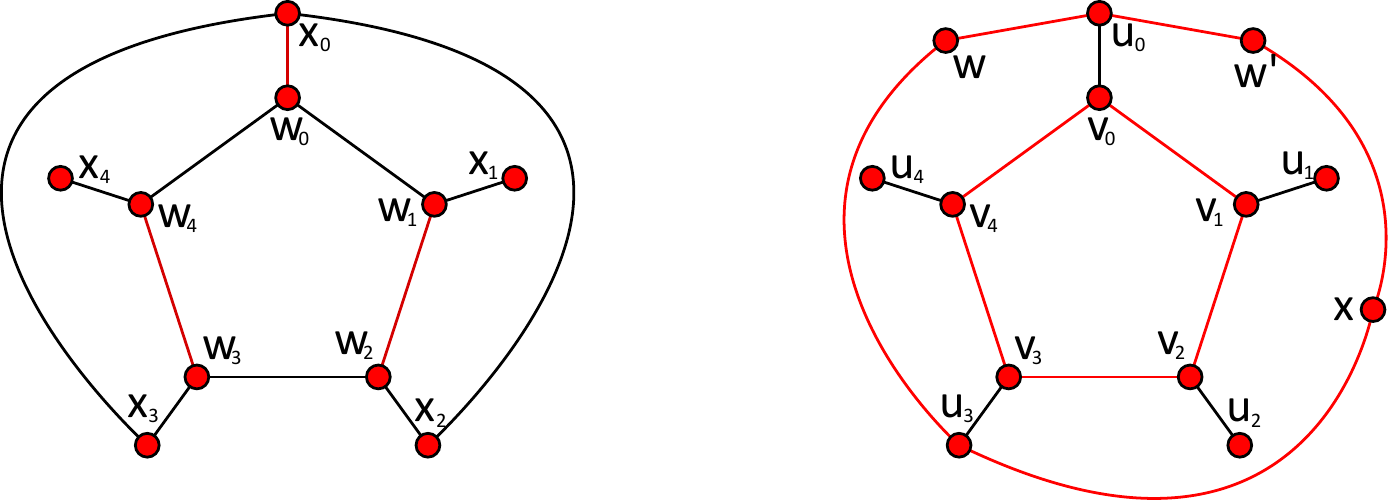}
\end{center}
\caption{The local situations around a $5$-cycle.}
\label{fig:girth5proof}
\end{figure}

We now show that the second possibility also cannot occur. Suppose to the contrary that each vertex is the internal vertex of one $2$-path lying on two $5$-cycles and of two $2$-paths lying on one $5$-cycle. Similarly as above we find that each vertex of $\Gamma$ lies on four $5$-cycles and that each vertex is incident to one edge lying on two $5$-cycles (we color such edges red) and to two edges each lying on three $5$-cycles (we color these black). In particular, $4n = 5s$, where $n$ is the order of $\G$ and $s$ is the number of $5$-cycles of $\G$. We clearly cannot have a black $5$-cycle since then $\Gamma$ is the Petersen graph in which all $2$-paths are on two $5$-cycles, and a similar argument as above shows that no black edge on a $5$-cycle can be surrounded by two red edges. It follows that each $5$-cycle of $\G$ consists of one red and four black edges. Since red edges are contained on two $5$-cycles (and there is $n/2$ of them) this implies $n = s$, contradicting $4n = 5s$. 
This finally completes Case 1.
\medskip

\noindent
{\bf Case 2:} the set $\{u_i \colon 1 \leq i \leq 5\}$ is an independent set.\\
Suppose first that $v_0$ (and thus any vertex of $\Gamma$) lies on more than one $5$-cycle. Since there are no edges between the vertices $u_i$ the only possibility is that either $u_0$ and $u_1$ or $u_0$ and $u_4$ have a common neighbor; without loss of generality assume $u_0$ and $u_1$ do. Since $v_3$ also lies on at least two $5$-cycles either $u_3$ and $u_2$ or $u_3$ and $u_4$ have a common neighbor. It is now easy to see that $\Gamma$ is the Dodecahedron graph.

We are thus left with the possibility that each vertex lies on a unique $5$-cycle, and so each vertex is incident to two edges lying on a unique $5$-cycle (we color such edges red) and to one edge that does not lie on a $5$-cycle (we color these black). Each $5$-cycle of $\G$ thus consist of five black edges. Now, let $C$, $v_i$ and $u_i$ be as at the beginning of this proof and let $w$ and $w'$ be the two red neighbors of $u_0$. Since there is no $5$-cycle through $u_0v_0$ none of $w, w'$ is adjacent to any of $u_1$ and $u_4$. In fact, none of $w, w'$ is adjacent to any of the $u_i$, $i \neq 0$. For, if this was the case, say $w \sim u_3$, then $u_3$ and $w'$ would have a common black neighbor, say $x$ (see the right part of Figure~\ref{fig:girth5proof}). But then the black neighbors $u_0$ and $u_3$ of the two red neighbors of $v_4$ would have a common red neighbor (namely $w$), and so the same should hold for the black neighbors $u_0$ and $u_2$ of the two red neighbors of $v_1$. However, as both $w$ and $w'$ already have both of their red neighbors, this is impossible.

It thus follows that the red $5$-cycle containing $u_0$ contains none of the vertices $u_i$, $i \neq 0$. Thus, each vertex of the $5$-cycle $C$ has precisely one neighbor outside $C$ and no two vertices of $C$ have a neighbor in the same $5$-cycle $C'$, different from $C$. Since the $5$-cycles of $\Gamma$ clearly form an imprimitivity block system for the automorphism group $\Aut(\Gamma)$ we can apply Theorem~\ref{the:VTtrunc}. The isomorphism $\Aut(\Gamma) \cong \Aut(\Lambda)$ follows from Corollary~\ref{cor:small_girth} and Theorem~\ref{the:VTfromAT}.
\end{proof}

\begin{remark}
It is well known that there exist infinitely many $5$-valent arc-transitive graphs (see for instance~\cite{AntHujKut15}), and so Theorem~\ref{the:girth5} implies that there exist infinitely many cubic vertex-transitive graphs of girth $5$.
\end{remark}

\end{document}